\documentclass[reqno,12pt]{amsart}

\usepackage{amsmath,amsthm,amssymb,amsfonts,amscd,comment}
\usepackage{xypic}

\theoremstyle{plain} 
\newtheorem{thm}{Theorem}[section]
\newtheorem*{thm*}{Theorem}

\newtheorem{lem}[thm]{Lemma}
\newtheorem{prop}[thm]{Proposition}

\newtheorem*{conj*}{Conjecture}

\newtheorem{dfn}[thm]{Definition}
	
\theoremstyle{definition}
\newtheorem{eg}[thm]{Example}
\newtheorem{prob}[thm]{Problem}

\newtheorem{question}[thm]{Question}

\newtheorem{rem}[thm]{Remark}	
\theoremstyle{remark}

\newtheorem*{pf}{Proof}

\numberwithin{equation}{section}
\def\NN{{\mathbb N}}
\def\ZZ{{\mathbb Z}}
\def\QQ{{\mathbb Q}}
\def\RR{{\mathbb R}}
\def\CC{{\mathbb C}}
\def\PP{{\mathbb P}}

\def\PP{{\mathbb P}}

\def\A{{\mathcal A}}
\def\B{{\mathcal B}}

\def\D{{\mathcal D}}
\def\E{{\mathcal E}}
\def\F{{\mathcal F}}

\def\O{{\mathcal O}}

\def\S{{\mathcal S}}
\def\T{{\mathcal T}}

\newcommand{\per}{\mathrm{per}}
\newcommand{\ch}{\mathrm{ch}}
\newcommand{\iso}{\xrightarrow{\sim}}
\newcommand{\Aut}{\mathrm{Aut}}

\newcommand{\dgcat}{\mathbf{dgcat}_k}
\newcommand{\hodgcat}{\mathbf{hodgcat}_k}
\newcommand{\RHom}{\mathbf{R}\mathcal{H}om}

\newcommand{\Hom}{{\rm Hom}}

\newcommand{\Ext}{{\rm Ext}}

\def \mf#1#2#3#4{
\xymatrix{
{#1}\  \ar@<0.4ex>[r]^{{#2}} & \ {#4}
\ar@<0.4ex>[l]^{{#3}}
}
}

\def \mfs#1#2#3#4{\!
\xymatrix@C=1,5em{{#1} \! \ar@<0.2ex>[r]^{{#2}} & \! {#4}
\ar@<0.2ex>[l]^{{#3}}
}
\!}

\def \mfl#1#2#3#4{
\xymatrix@C=2.6em{{#1}\  \ar@<0.4ex>[r]^{{#2}} &\  {#4}
\ar@<0.2ex>[l]^{{#3}}
}
}

\def \mfss#1#2#3#4{\!
\xymatrix@C=1.5em{{#1} \ar@<0.3ex>[r]^{{#2}} & {#4}
\ar@<0.3ex>[l]^{{#3}}
}
\!}
\begin{document}
\title{Hochschild entropy and Categorical entropy}
\author{Kohei Kikuta}
\address{Department of Mathematics, Chuo University, Tokyo, 112-0003, Japan}
\email{kikuta@math.chuo-u.ac.jp}
\author{Genki Ouchi}
\address{Graduate School of Mathematics, Nagoya University, 
Furocho, Chikusaku, Nagoya, Japan, 464-8602}
\email{genki.ouchi@math.nagoya-u.ac.jp}
\begin{abstract}
We study the categorical entropy and counterexamples to Gromov--Yomdin type conjecture via homological mirror symmetry of K3 surfaces established by Sheridan--Smith. 
We introduce asymptotic invariants of quasi-endofunctors of dg categories, called the Hochschild entropy. 
It is proved that the categorical entropy is lower bounded by the Hochschild entropy. 
Furthermore, motivated by Thurston's classical result, we prove the existence of a symplectic Torelli mapping class of positive categorical entropy. 
We also consider relations to the Floer-theoretic entropy. 
\end{abstract}
\maketitle
\section{Introduction}
\subsection{Categorical entropy and Gromov--Yomdin type equality}
Let $X$ be a smooth projective variety over $\CC$, and $f\in\Aut(X)$ an automorphism of $X$. 
There is an important asymptotic invariant $h_{\mathrm{top}}(f)\in\RR_{\geq0}$ called the {\it topological entropy} of $f$,
which measures the complexity of the dynamical system $(X,f)$. 
As a categorical analogue of the topological entropy, for a triangulated category $\T$,  
Dimitrov--Haiden--Katzarkov--Kontsevich introduced the {\it categorical entropy} $h_{\mathrm{cat}}(\Phi)\in\RR_{\geq0}$ of an autoequivalence $\Phi\in\Aut(\T)$ \cite{DHKK}. 
These two entropies agree in the following cases: 
\begin{itemize}
\item
\cite{KT}: 
Let $\D^b(X)$ be the bounded derived category of coherent sheaves on $X$. 
For an automorphism $f$ of $X$, 
its pull-back $f^*\in\Aut(\D^b(X))$ satisfies 
$h_{\mathrm{top}}(f)=h_{\mathrm{cat}}(f^*).$
\item
\cite{DHKK}: 
Let $\Sigma_g$ be a connected oriented closed surface of genus $g\ge2$, and $\D\mathrm{Fuk}(\Sigma_g)$ the derived Fukaya category of $\Sigma_g$. 
For a pseudo-Anosov mapping class $\phi\in \mathrm{MCG}(\Sigma_g)$, its induced autoequivalence $\phi_*\in\Aut(\D\mathrm{Fuk}(\Sigma_g))$ satisfies
$h_{\mathrm{top}}(\phi)=h_{\mathrm{cat}}(\phi_*).$ 
\end{itemize}

In general, it is hard to compute the topological entropy of homeomorphisms, whereas that of automorphisms of algebraic varieties can be computed as the spectral radius of the induced action on the cohomology due to Gromov--Yomdin. 
\begin{thm}[Gromov--Yomdin \cite{Gro1,Gro2,Yom}]
For $f\in\Aut(X)$, one has 
\[
h_{\mathrm{top}}(f)=\log\rho(H^*(f)).
\]
\end{thm}
\noindent
It is therefore natural to consider the categorical analogue of the Gromov--Yomdin theorem, that is the agreement between the categorical entropy and the spectral radius of the induced action on the numerical Grothendieck group of the derived category:
\[
h_{\mathrm{cat}}(\Phi)=\log\rho(K_{{\rm num}}(\Phi))~\text{ for }\Phi\in\Aut(\D^b(X)).
\]
This Gromov--Yomdin type equality is conjectured by the first author and Takahashi \cite{KT}, and proved for any autoequivalence of curves, abelian surfaces, simple abelian varieties and varieties with the ample (anti-)canonical sheaf \cite{Kik,Yos,KT}. 
This conjecture is however not true in general. 
Y.-W. Fan found, as the first counterexample, some autoequivalence of any Calabi--Yau hypersurfaces of even dimension greater than two \cite{Fan}. 
In the sequel the second author and Mattei construct counterexamples for surfaces \cite{Ouc,Mat}. 

\subsection{Counterexamples via Mirror symmetry}
In algebro-geometric viewpoint, it is natural to expect the categorical analogue of Gromov--Yomdin theorem. 
On the other hand, Thurston proved the existence of mapping classes of $\Sigma_g$ with positive topological entropy, trivially acting on the cohomology, hence in the Torelli subgroup of ${\rm MCG}(\Sigma_g)$ \cite{Thu}. 
Combining the above result by \cite{DHKK}, in symplecto-geometric viewpoint, we obtain an autoequivalence of $\D\mathrm{Fuk}(\Sigma_g)$ with positive categorical entropy, which is induced by a Torelli mapping class. 
We therefore expect the following. 
\begin{prob}\label{prob}
\begin{enumerate}
\item
Does there exist a symplectic Torelli mapping class of positive categorical entropy for Calabi--Yau manifolds?
\item
Does a symplectic Torelli mapping class as in (i) give a counterexample to Gromov--Yomdin type conjecture via the homological mirror symmetry?

\end{enumerate}
\end{prob}
These kinds of problems are one of Fan's motivations to find a counterexample to the conjecture as in \cite[Section 1]{Fan}. 

\subsection{Floer-theoretic entropy}
To study dynamics of mapping classes of closed surfaces via Floer theory, Fel'shtyn introduced the asymptotic invariant defined to be the exponential growth rate of the dimensions of the fixed-point Floer cohomology of $\phi\in{\rm MCG}(\Sigma_g)$
\[
h_{\mathrm{Floer}}(\phi)
:=\limsup_{n\to\infty}\frac{1}{n}\log\dim_\Lambda HF^*(\phi^n),
\]
see \cite{Fel2}, and Smith also considered the same invariant which he called the {\it Floer-theoretic entropy} \cite{Smi1}. 
For pseudo-Anosov mapping class, its Floer-theoretic entropy coincides with the topological entropy \cite{Fel4}. 

It is interesting to compare the Floer-theoretic entropy and the categorical entropy, which is a large motivation for this paper. 

\subsection{Main results}
Motivated by the comparison between the Floer-theoretic entropy and the categorical entropy,
we introduce the asymptotic invariants given by the exponential growth rate of dimensions of Hochschild (co)homology of quasi-functors, which we call the {\it Hochschild (co)homological entropy}: 
Let $\A$ be a smooth proper differential graded (dg) category, and $\per_{\mathrm{dg}}(\A)$ its perfect derived dg category. 
For a quasi-endofunctor $\widetilde{\Phi}$ of $\per_{\mathrm{dg}}(\A)$, 
we define
\begin{eqnarray*}
h_{HH^*}(\widetilde{\Phi})&:=&\limsup_{n \to \infty}\frac{1}{n}\log \mathrm{dim}_kHH^*(\widetilde{\Phi}^n),\\
h_{HH_*}(\widetilde{\Phi})&:=&\limsup_{n \to \infty}\frac{1}{n}\log \mathrm{dim}_kHH_*(\widetilde{\Phi}^n). 
\end{eqnarray*}
These two entropies agree for Calabi--Yau dg categories. 
As Smith observed in \cite[Corollary 3.12]{Smi1}, 
the fixed-point Floer cohomology is sometimes isomorphic to the Hochschild cohomology. 
Therefore, in the first step, we replace the Floer-theoretic entropy by the Hochschild entropy. 

In this paper, we firstly compare the Hochschild entropy and the categorical entropy.
\begin{thm}[Theorem \ref{HH-cat}]
For a quasi-endofunctor $\widetilde{\Phi}$ of $\per_{\mathrm{dg}}(\A)$ and $\Phi:=H^0(\widetilde{\Phi})$, 
we have 
\[
h_{HH^*}(\widetilde{\Phi}) \leq h_{\mathrm{cat}}(\Phi)
\text{ and }
h_{HH_*}(\widetilde{\Phi}) \leq h_{\mathrm{cat}}(\Phi).
\]
\end{thm}
We furthermore expect the agreement between the Floer-theoretic entropy and the Hochschild entropy for symplectic K3 surfaces, see Section 6 where we give Question \ref{Floer-HH} and propose how to prove it.  

We then consider Problem \ref{prob} for K3 surfaces. 
Let $\check{X}_0$ be the complex K3 surface called {\it mirror quartic}, defined as the crepant resolution of the quotient of the Fermat quartic hypersurface in $\PP^3_{\CC}$ by $\ZZ/4\times\ZZ/4$. 
Sheridan--Smith proved the homological mirror symmetry for Greene--Plesser mirrors \cite{Smi1}. 
Combining with a result in \cite{Smi2}, as a special case of Greene--Plesser mirrors, one has the mirror equivalence for the mirror quartic (see Section 5 for details): 
There exists a K\"ahler form $\omega$ on $\check{X}_0$ and a quartic K3 surface $X$
of Picard rank one, which satisfies a quasi-equivalence of
$A_{\infty}$ categories
\[
{\mathfrak D}{\rm Fuk}(\check{X}_0,\omega)
\simeq\per_{\mathrm{dg}}(X). 
\]
Here ${\mathfrak D}{\rm Fuk}(\check{X}_0,\omega)$ is the split-closed triangulated envelope of the Fukaya category of $(\check{X}_0,\omega)$ and 
$\per_{\mathrm{dg}}(X)$ is the dg category of perfect complexes on $X$, which is a dg enhancement of $\D^b(X)$. 
Bayer--Bridgeland proved Bridgeland conjecture for complex algebraic K3 surfaces of Picard rank one \cite{BB}
and, as a corollary, 
identified the subgroup of Calabi--Yau autoequivalences of $\D^b(X)$, modulo even shifts, and the stacky fundamental group of K\"ahler moduli space of $X$. 
Based on \cite{SS1} and \cite{BB}, 
Sheridan--Smith furthermore developed the relation between the symplectic mapping class groups and the autoequivalence groups \cite{SS2}. 
Using their results, we prove that the counterexample to Gromov--Yomdin type conjecture on $X$ in \cite{Ouc} induces a symplectic Torelli mapping class on $(\check{X_0},\omega)$ as in Problem \ref{prob} (i).

\begin{thm}[Theorem \ref{exist-sMCG}]\label{intro-exist-sMCG}
Let $(\check{X_0},\omega)$ be the mirror quartic. 
Then there exists a symplectic Torelli mapping class $\phi$ of $\check{X}_0$ satisfying
\[
h_{\mathrm{cat}}(\phi)>\log\rho(H^2(\phi))=0. 
\]
\end{thm}
By the construction, the symplectic Torelli mapping class $\phi$ in Theorem \ref{intro-exist-sMCG} answers Problem \ref{prob} (ii). 

The homological mirror symmetry is a quasi-equivalence of some $A_{\infty}$ categories over the universal Novikov field $\Lambda$, not over $\CC$. 
To estimate the Hochschild and categorical entropy over $\Lambda$ via complex models, 
we establish the base change formula for field extensions in general settings, see Theorem \ref{base-change}. 

\bigskip
\noindent
{\bf Acknowledgements.}
K.K. is indebted to Ivan Smith for explaining a strategy to prove Question \ref{Floer-HH}. 
The authors would like to thank his Atsushi Takahashi for useful discussions. 
K.K. is partially supported by JSPS KAKENHI Grant Number 20K22310. 
G.O is supported by JSPS KAKENHI Grant Number 19K14520 and Interdisciplinary Theoretical and Mathematical Sciences Program (iTHEMS) in RIKEN.

\bigskip
\noindent
{\bf Notation and Convention.} 
Throughout this paper, for a smooth projective variety $X$ over a field $k$, the bounded derived category of coherent sheaves on $X$ is denoted by $\D^b(X)$. 
For an object $\E \in \D^b(X\times Y)$, the Fourier--Mukai functor $\Phi_\E$ associated with the Fourier--Mukai kernel $\E$ is defined by 
\[\Phi_\E: \D^b(X) \to \D^b(Y), E \mapsto \mathbf{R}p_*(q^*E \otimes^{\mathbf{L}} \E), \]
where $p:X \times Y \to Y$ and $q: X \times Y \to X$ are projections. 
If $\Phi:\D^b(X) \to \D^b(Y)$ is an equivalence, there uniquely exists a Fourier--Mukai kernel $\E \in \D^b(X \times Y)$ such that $\Phi \simeq \Phi_\E$ by {\cite[Theorem 2.19]{Orl1}}. 
For $\E \in \D^b(X\times Y)$ and $\F \in \D^b(Y \times Z)$, 
we define the convolution product 
\[\F \circ \E:=\mathbf{R}\pi_{{XZ}*}(\pi^*_{XY}\E \otimes^{\mathbf{L}}\pi^*_{YZ}\F) \in \D^b(X \times Z), \]
where $\pi_{XY}, \pi_{XZ}$ and $\pi_{YZ}$ are projections from $X \times Y \times Z$ to $X \times Y$, $X \times Z$ and $Y \times Z$ respectively.
By \cite[Proposition 5.10]{Huy1}, we have $\Phi_\F \circ \Phi_\E \simeq \Phi_{\F \circ \E}$. 

Let $\per_{\mathrm{dg}}(X)$ be the dg category of perfect complexes on $X$. Since $X$ is smooth, we have the exact equivalence $H^0(\per_{\mathrm{dg}}(X)) \simeq \D^b(X)$, that is, $\per_{\mathrm{dg}}(X)$ is a dg enhancement of $\D^b(X)$. 

For a subring $K\subset\CC$ and a finitely generated free $K$-module $V$, define the spectral radius of a $K$-endomorphism $f:V \to V$ by $\rho(f):=\rho(f_\CC)$.
We set $\log0:=-\infty$.
\section{Entropy}
\subsection{Categorical entropy}
Let $\T$ be a triangulated category over a field $k$. For an object $E \in \T$, let $\langle E \rangle$ be the smallest full triangulated subcategory that contains $E$ and is closed under taking direct summands.

\begin{dfn}
Let $E, G$ be objects in $\T$ such that $E \in \langle G \rangle$. The complexity $\delta_t(G,E)$ of $E$ with respect to $G$ is the function  $\delta_t(G,E):\RR \to [0,+\infty)$ defined by
\begin{equation*}
\delta_t(G,E)
:=
\inf\left\{
\displaystyle\sum_{i=1}^p e^{n_i t}~
\middle
|~
\begin{xy}
(0,5) *{0}="0", (20,5)*{E_{1}}="1", (30,5)*{\dots}, (40,5)*{E_{p-1}}="k-1", (60,5)*{E\oplus E'}="k",
(10,-5)*{G[n_{1}]}="n1", (30,-5)*{\dots}, (50,-5)*{G[n_{p}]}="nk",
\ar "0"; "1"
\ar "1"; "n1"
\ar@{.>} "n1";"0"
\ar "k-1"; "k" 
\ar "k"; "nk"
\ar@{.>} "nk";"k-1"
\end{xy}
\right\}.
\end{equation*}
\end{dfn}

We collect fundamental inequalities.

\begin{prop}[{\cite[Proposition 2.2]{DHKK}}]\label{complexity ineq}
The following are hold.
\begin{itemize}
\item[(1)] Let $G_1, G_2, G_3$ be objects in $\T$ such that $G_2, G_3 \in \langle G_1 \rangle, G_3 \in \langle G_2 \rangle$. We have the inequality
\[\delta_t(G_1,G_3)\leq \delta_t(G_1,G_2)\delta_t(G_2,G_3). \]
\item[(2)] Let $E, G$ be an object in $\T$ such that $E \in \langle G \rangle$. For a non-zero exact functor $\Phi: \T \to \T'$ to a triangulated category $\T'$, we have the inequality
\[\delta_t(\Phi(G),\Phi(E)) \leq \delta_t(G,E).\]
\end{itemize}
\end{prop}
An object $G \in \T$ is a split generator if $\T=\langle G \rangle$.
We recall the definition of the categorical entropy of autoequivalences of $\T$.

\begin{dfn}[{\cite[Definition 2.4, Lemma 2.5]{DHKK}}]
Let $G$ and $G'$ be split generators of $\T$. For an autoequivalence $\Phi:\T \to \T$,
the categorical entropy $h_t(\Phi)$ is defined by 
\begin{equation*}
h_{t}(\Phi):=\displaystyle\lim_{n\rightarrow\infty}\frac{1}{n}\log \delta_t(G,\Phi^{n}(G')).
\end{equation*}
The limit exists in $[-\infty,+\infty)$ for any $t \in \RR$ and is independent of the choices of $G$ and $G'$. 
Denote the value $h_0(\Phi)\geq 0$ at $t=0$ by $h_{\mathrm{cat}}(\Phi)$. 
\end{dfn}

Let $\A$ be a smooth proper differential graded (dg) category over a field $k$. 
Denote the perfect derived category of $\A$ by $\per(\A)$. Since $\A$ is proper, the triangulated category $\per(\A)$ is Ext-finite over $k$, and there is a split generator of $\per(\A)$. 
Let $\T$ be a triangulated category equivalent to $\per(\A)$ for some smooth proper dg category $\A$ over a field $k$. 
In this case, the categorical entropy can be computed by the following theorem.


\begin{thm}[{\cite[Theorem 2.6]{DHKK}}]\label{entropy ext}
Let $G$ and $G'$ be split generators of $\T$. For an exact endofunctor $\Phi:\T \to \T$ such that $\Phi^n$ is not zero for any $n>0$, we have
\begin{equation*}
h_{t}(\Phi)=\displaystyle\lim_{n\rightarrow\infty}\frac{1}{n}\log \delta'_t(G,\Phi^{n}(G')).
\end{equation*}
Here, $\delta'_t(E,F):=\sum_{i \in \ZZ}\dim_k\Ext^i(E,F)e^{-it}$ for objects $E,F \in \T$.
\end{thm}

For $E,F \in \T$, we define the Euler characteristic $\chi(E,F)$ by \[\chi(E,F):=\sum_{i \in \ZZ}(-1)^i\dim_k\Ext^i(E,F).\] 
It induces the bilinear form $\chi(-,-)$ on the Grothendieck group $K_0(\T)$ of $\T$. The numerical Grothendieck group $K_{\mathrm{num}}(\T)$ of $\T$ is the quotient of $K_0(\T)$ by the radical of $\chi$. 
The group $K_{\mathrm{num}}(\T)$ is a finitely generated free abelian group by the Hirzebruch--Riemann--Roch theorem \cite{Lun}, \cite{Shk}.
An autoequivalence $\Phi : \T \iso \T$ induces the linear map 
$K_{\mathrm{num}}(\Phi):K_{\mathrm{num}}(\T) \iso K_{\mathrm{num}}(\T)$. 

The following is an analogue of Yomdin's inequality \cite{Yom}.

\begin{prop}[{\cite[Theorem 2.13]{KST}}]
Let $\Phi$ be an autoequivalence of $\T$. 
Then we have the inequality
\begin{equation}\label{cat-spec}
h_{\mathrm{cat}}(\Phi) \geq \log \rho(K_{\mathrm{num}}(\Phi)). 
\end{equation}
\end{prop}

\subsection{Hochschild entropy}


Let $\dgcat$ be the category of dg categories over a field $k$, 
whose morphisms are dg functors.
Denote the localization of $\dgcat$ along quasi-equivalences by $\hodgcat$. The internal hom defines the functor
\[
\RHom(-,-):\hodgcat^{\mathrm{op}} \times \hodgcat \to \hodgcat.
\]
The tensor product defines the functor 
\[
-\otimes^\mathbf{L}-: \hodgcat \times \hodgcat \to \hodgcat.
\]
For dg categories $\A,\B$, we have the natural bijection \cite{Toe}
\[
\hodgcat(\A,\B) \simeq  \mathrm{Iso}(H^0(\RHom(\A,\B))).
\]
Take a smooth proper dg category $\A$. 
Let $\D$ be the perfect derived dg category $\per_{\mathrm{dg}}(\A)$ of $\A$, 
which is a dg enhancement of $\per(\A)$. 
By \cite{Toe}, we have the isomorphism \[\RHom(\D,\D) \simeq \per_{\mathrm{dg}}(\A^{\mathrm{op}} \otimes^\mathbf{L} \A)\]
in $\hodgcat$. By the equivalence $H^0(\RHom(\D,\D)) \simeq \per(\A^{\mathrm{op}} \otimes^\mathbf{L} \A)$,
we have the object $\S^{-1}_{\D} \in H^0(\RHom(\D,\D))$ which is corresponding to the inverse dualizing complex of $\A$ in $\per(\A^{\mathrm{op}} \otimes^\mathbf{L} \A)$.
An object of $H^0(\RHom(\D,\D))$ (or an element in $\mathrm{Iso}(H^0(\RHom(\D,\D)))$) is called a {\it quasi-endofunctor} of $\D$.
For $\widetilde{\Phi}\in H^0(\RHom(\D,\D))$, we have the exact functor $H^0(\widetilde{\Phi}):\per(\A) \to \per(\A)$. Then we obtain the functor 
\[
H^0:~H^0(\RHom(\D,\D)) \to \mathrm{Fun}(\per(\A),\per(\A)).
\]

We define Hochschild (co)homology groups of quasi-endofunctors of $\D$ as follows.


\begin{dfn}\label{HH}
Take an integer $i \in \ZZ$. Let $\widetilde{\Phi}\in H^0(\RHom(\D,\D))$ be a quasi-functor. 

The $i$-th Hochschild cohomology group $HH^i(\widetilde{\Phi})$ of $\widetilde{\Phi}$ is defined by 
\[HH^i(\widetilde{\Phi}):=\Ext^i_{H^0(\RHom(\D,\D))}(\widetilde{\Phi}, \mathrm{id}_\D).\]
Similarly, the $i$-th Hochschild homology group $HH_i(\widetilde{\Phi})$ of $\widetilde{\Phi}$ is defined by 
\[HH_i(\widetilde{\Phi}):=\Ext^{-i}_{H^0(\RHom(\D,\D))}(\S^{-1}_\D, \widetilde{\Phi}).\]
We put 
\[HH^*(\widetilde{\Phi}):=\bigoplus_{i \in \ZZ}HH^i(\widetilde{\Phi}),\]
\[HH_*(\widetilde{\Phi}):=\bigoplus_{i \in \ZZ}HH_i(\widetilde{\Phi}).\]

\end{dfn}

\begin{rem}
By $H^0(\RHom(\D,\D)) \simeq \per(\A^{\mathrm{op}} \otimes^\mathbf{L} \A)$,   $HH^*(\widetilde{\Phi})$ and $HH_*(\widetilde{\Phi})$ are finite dimensional $k$-linear spaces for any quasi-functor $\widetilde{\Phi}\in H^0(\RHom(\D,\D))$.
\end{rem}
Definition \ref{HH} is the generalization of the usual Hochschild (co)homology group of $\D$.

\begin{rem}\label{rem:HHhom}
The $i$-th Hochschild cohomology group $HH^i(\D)$ of $\D$ is given by 
\[HH^i(\D)=HH^i(\mathrm{id}_\D).\]
Similarly, the $i$-th Hochschild homology group $HH_i(\D)$ of $\D$ is given by 
\[HH_i(\D)=HH_i(\mathrm{id}_\D).\]
\end{rem}

We define the Hochschild (co)homological entropy of quasi-endofunctors of $\D$.

\begin{dfn}\label{HH entropy}
Let $\widetilde{\Phi}\in H^0(\RHom(\D,\D))$ be a quasi-endofunctor of $\D$.
We define the Hochschild cohomological entropy of $\widetilde{\Phi}$ by
\[h_{HH^*}(\widetilde{\Phi}):=\limsup_{n \to \infty}\frac{1}{n}\log \mathrm{dim}_kHH^*(\widetilde{\Phi}^n). \]
Similarly, we define the Hochschild homological entropy of $\widetilde{\Phi}$ by
\[h_{HH_*}(\widetilde{\Phi}):=\limsup_{n \to \infty}\frac{1}{n}\log \mathrm{dim}_kHH_*(\widetilde{\Phi}^n). \]
\end{dfn}

\begin{thm}\label{HH-cat}
For $\widetilde{\Phi}\in H^0(\RHom(\D,\D))$ and $\Phi:=H^0(\widetilde{\Phi})$, we have 
\begin{equation}
h_{HH^*}(\widetilde{\Phi}) \leq h_{\mathrm{cat}}(\Phi)
\end{equation}
and
\begin{equation}\label{ineq-HH-cat-t}
h_{HH_*}(\widetilde{\Phi}) \leq h_{\mathrm{cat}}(\Phi).
\end{equation}
\end{thm}
\begin{proof}
Let $G \in \D$ be a split generator of $\D$.
Since $\D$ is smooth and proper,\\
$G \otimes \mathbf{R}\Hom(G,-)$ is a split generator of $H^0(\RHom(\D,\D))$. 
Firstly, we have
\begin{eqnarray*}
\mathrm{dim}_kHH^*(\widetilde{\Phi}^n) &=& \delta_0(k, \mathbf{R}\Hom(\widetilde{\Phi}^n, \mathrm{id}_\D))\\
&=& \delta_0(k, \mathbf{R}\Hom(\widetilde{\Phi}^n, \mathrm{id}_\D)^*)\\
&\leq& \delta_0(k, \mathbf{R}\Hom(G \otimes \mathbf{R}\Hom(G,-), \mathrm{id}_\D)^*)\\
&& \delta_0(\mathbf{R}\Hom(G \otimes \mathbf{R}\Hom(G,-), \mathrm{id}_\D)^*, \mathbf{R}\Hom(\widetilde{\Phi}^n, \mathrm{id}_\D)^*)\\
&\leq&  \delta_0(k, \mathbf{R}\Hom(G \otimes \mathbf{R}\Hom(G,-), \mathrm{id}_\D)^*) \\
&& \delta_0(G \otimes \mathbf{R}\Hom(G,-), \widetilde{\Phi}^n).
\end{eqnarray*}
Here, the first (resp. second) inequality is deduced from Proposition \ref{complexity ineq} (1) (resp. (2)). 

Moreover, we obtain
\begin{eqnarray*}
\delta_0(G \otimes \mathbf{R}\Hom(G,-), \widetilde{\Phi}^n)&\leq&\delta_0(G \otimes \mathbf{R}\Hom(G,-), \widetilde{\Phi}^n \circ (G \otimes \mathbf{R}\Hom(G,-)))\\
&&\delta_0(\widetilde{\Phi}^n \circ (G \otimes \mathbf{R}\Hom(G,-)), \Phi^n)\\
&\leq& \delta_0(G, H^0(\Phi)^n(G)) \delta_0(G \otimes \mathbf{R}\Hom(G,-), \mathrm{id}_\D).
\end{eqnarray*}
The first inequality is deduced from Proposition \ref{complexity ineq} (1).
The second inequality is deduced from the isomorphism 
$\widetilde{\Phi}^n \circ (G \otimes \mathbf{R}\Hom(G,-)) \simeq H^0(\Phi)^n(G) \otimes \mathbf{R}\Hom(G,-)$ and Proposition \ref{complexity ineq} (2).
Therefore, we have the inequality 
\[h_{HH^*}(\widetilde{\Phi}) \leq h_{\mathrm{cat}}(H^0(\Phi)).\]
To show $h_{HH_*}(\widetilde{\Phi}) \leq h_{\mathrm{cat}}(\Phi)$, we can apply similar computations.
\end{proof}

\begin{rem}
The inequality (\ref{ineq-HH-cat-t}) has already been proved via a different method by the first author and Atsushi Takahashi. 
\end{rem}


\begin{rem}
If $\D$ is a Calabi--Yau dg category, we have 
\[h_{HH^*}(\widetilde{\Phi})=h_{HH_*}(\widetilde{\Phi})\]
for any $\widetilde{\Phi}\in H^0(\RHom(\D,\D))$.
\end{rem}
We have the following question.
\begin{question}\label{HHHH}
Let $\widetilde{\Phi}\in H^0(\RHom(\D,\D))$ and $\Phi:=H^0(\widetilde{\Phi})$. 
When do we have 
\[
h_{HH^*}(\widetilde{\Phi})=h_{HH_*}(\widetilde{\Phi})=h_{\mathrm{cat}}(H^0(\Phi))?
\]
After the first version of this article was submitted on arXiv, Atsushi Takahashi pointed out that Question \ref{HHHH} is false in general: for a phantom dg category $\D\neq 0$, 
\[
h_{HH^*}({\rm id}_\D)=0\neq -\infty=h_{HH_*}({\rm id}_\D)
\]
by Remark \ref{rem:HHhom}. 
For example, (a dg enhancement of) an admissible subcategory of the derived category of the classical Godeaux surface is smooth, proper and phantom by \cite[Theorem 1.1]{BvBS} and \cite[Remark 5.17]{BLS}. 
\end{question}

For $\widetilde{\Phi}\in H^0(\RHom(\D,\D))$, we obtain the induced linear map
\[\widetilde{\Phi}^{HH_*}:HH_*(\D) \to HH_*(\D).\]
When $k=\CC$, we compare the Hochschild homological entropy $h_{HH_*}(\widetilde{\Phi})$ with the spectral radius $\rho(\widetilde{\Phi}^{HH_*})$.

\begin{prop}\label{HH-spec-rad-t}
Assume that $k=\CC$ and $H_{\mathrm{odd}}(\D)=0$. 
For $\widetilde{\Phi}\in H^0(\RHom(\D,\D))$, we have the inequality
\[h_{HH_*}(\widetilde{\Phi}) \geq \log \rho(\widetilde{\Phi}^{HH_*}).\]
\end{prop}
\begin{proof}
By the Lefschetz fixed point theorem for Hochschild homology \cite{Lun} and the assumption, 
we obtain 
\begin{eqnarray*}
\dim_kHH_*(\widetilde{\Phi}^n) &\geq& \sum_{i \in \ZZ} (-1)^i\dim_kHH_i(\widetilde{\Phi}^n)    \\
&=& \mathrm{tr}\Bigl(\bigl(\widetilde{\Phi}^{HH_*}\bigr)^n\Bigr). 
\end{eqnarray*}
Note that we have the equality
 \[\mathrm{exp}\Bigl(\sum_{i \in \ZZ}\frac{\mathrm{tr}((\widetilde{\Phi}^{HH_*})^n)}{i}z^i\Bigr)=\mathrm{det}(1-\widetilde{\Phi}^{HH_*}z)^{-1}. \]
 Computing the radius of convergence of the above function, we have
\[
\limsup_{n \to \infty} \mathrm{tr}\Bigl(\bigl(\widetilde{\Phi}^{HH_*}\bigr)^n\Bigr)^{1/n}=\rho(\widetilde{\Phi}^{HH_*}). 
\]
 Hence, we have proved the desired inequality.
\end{proof}
\begin{rem}
The proof of Proposition \ref{HH-spec-rad-t} is almost the same as that of \cite[Lemma 2.8]{DHKK}, which is used to compare $h_{{\rm cat}}(\Phi)$ with $\log \rho(\widetilde{\Phi}^{HH_*})$. 
\end{rem}


We interpret Definition \ref{HH} in terms of Fourier--Mukai kernels.
Let $X$ be a smooth projective variety over $k$. 
The dg category $\per_{\mathrm{dg}}(X)$ is isomorphic to the perfect derived dg category of some smooth proper dg category in $\hodgcat$ \cite[Lemma 3.27]{TV}. 
By {\cite[Theorem 8.15]{Toe}}, there is the natural isomorphism \[\RHom(\per_{\mathrm{dg}}(X),\per_{\mathrm{dg}}(X)) \simeq \per_{\mathrm{dg}}(X \times X) \]
in $\hodgcat$. Via the equivalence $H^0(\RHom(\per_{\mathrm{dg}}(X),\per_{\mathrm{dg}}(X))) \simeq \D^b(X \times X)$ of triangulated categories, the functor 
\[
H^0: H^0(\RHom(\per_{\mathrm{dg}}(X),\per_{\mathrm{dg}}(X))) \to \mathrm{Fun}(\D^b(X), \D^b(X))
\]
is corresponding to $\Phi_{(-)}: \D^b(X \times X) \to \mathrm{Fun}(\D^b(X), \D^b(X))$.
The corresponding quasi-functor to $\E\in\D^b(X\times X)$ is denoted by $\widetilde{\Phi_\E}$. 

\begin{rem}\label{HH via FM}
Let $\E \in \D^b(X \times X)$ be an object.
By {\cite[Theorem 8.15]{Toe}}, we have isomorphisms
\[HH^i(\widetilde{\Phi_\E}) \simeq \Ext^i(\E,\Delta_*\mathcal{O}_X[\dim X])\]
\[HH_i(\widetilde{\Phi_\E}) \simeq \Ext^{-i}(\Delta_*\omega^{-1}_X[-\dim X], \E)\]
for an integer $i \in \ZZ$,
where $\Delta: X\hookrightarrow  X \times X$ is the diagonal embedding.
\end{rem}
\subsubsection{$A_\infty$ version}\label{A_infty-ver}
We here consider slight generalizations for cohomologically unital (c-unital) $A_\infty$ functors (\cite[Ch.1 (2e)]{Sei}). 
Let $\B$ be a c-unital $A_\infty$ category over a field $k$. 
We denote the $A_\infty$ category of c-unital $A_\infty$ functors from $\B$ to $\B$ by $fun(\B,\B)$, 
and call an object in $fun(\B,\B)$  {\it (c-unital) $A_\infty$ endofunctor} of $\B$. 
We define the full subcategory $fun(\B,\B)^{qe}$ of $fun(\B,\B)$ consisting of quasi-equivalent $A_\infty$ endofunctors of $\B$. 
Define the group 
\begin{equation}\label{def-aut}
\Aut(\B):=\mathrm{Iso}(H^0(fun(\B,\B)^{qe})). 
\end{equation}

Suppose that $\B$ is quasi-equivalent to the perfect derived dg category $\D$ of a smooth proper (strictly unital) dg category. 
For $F\in\Aut(\B)$ (or more rigorously, $F\in H^0(fun(\B,\B))$), 
the {\it Hochschild cohomology} of $F$ is defined by
\begin{equation*}
HH^i(F)
:=\Ext^i_{H^0(fun(\B,\B))}(F, \mathrm{id}_\B)
(\simeq H^i(hom_{fun(\B,\B)}(F, \mathrm{id}_\B))). 
\end{equation*}
Canonaco--Ornaghi--Stellari proved the equivalence of categories
\[
\hodgcat\simeq \mathbf{hoA_\infty cat^c}_k
\]
and the isomorphism
\begin{equation}\label{COS}
\RHom(\D,\D)\simeq fun(\D,\D)(\simeq fun(\B,\B))
\end{equation}
in $\mathbf{hoA_\infty cat^c}_k$ (\cite[Theorem A,Theorem B, Corollary 2.6]{COS}), 
where $\mathbf{hoA_\infty cat^c}_k$ is the localization (with respect to quasi-equivalences) of the category $\mathbf{A_\infty cat^c}_k$ of c-unital $A_\infty$ categories over a field $k$, whose morphisms are c-unital $A_\infty$ functors. 
Therefore the above definition is compatible with Definition \ref{HH}. 
In the same manner as dg case, we can define the {\it  Hochschild cohomological entropy} of c-unital $A_\infty$ endofunctors. 
\begin{dfn}
For $F\in\Aut(\B)$, we define 
\[
h_{HH^*}(F):=\limsup_{n \to \infty}\frac{1}{n}\log \mathrm{dim}_kHH^*(F^n),
\]
which we call the {\it Hochschild cohomological entropy} of $F$. 
\end{dfn}
This definition is compatible with Definition \ref{HH entropy} by (\ref{COS}). 

\subsection{Base change by field extensions}
Let $k \subset K$ be an extension of fields. 
For a smooth projective variety $X$ over $k$, let $X_K:=X \times _k \mathrm{Spec }(K)$ be its base change. 
Taking the pullback with respect to the natural morphism $X_K \to X$, we have the exact functor
$(-)_K: \D^b(X) \to \D^b(X_K)$.

There are split generators of $\D^b(X)$ and $\D^b(X_K)$ which are compatible with the base change. 
\begin{rem}\label{split generator}
 Let $\mathcal{O}_X(1)$ be a very ample line bundle on $X$. 
 Put $\mathcal{O}_{X_K}(1):=\mathcal{O}_X(1)_K$. Then $\mathcal{O}_{X_K}(1)$ is a very ample line bundle on $X_K$. 
 By \cite[Theorem 4]{Orl2}, the object $G_p:=\bigoplus_{i=p}^{p+\dim X}\mathcal{O}_X(i)$ 
 is a split generator of $\D^b(X)$ for any integer $p \in \ZZ$. 
Thus the object $(G_{p})_K$ is also a split generator of $\D^b(X_K)$.
\end{rem}

The following are fundamental properties for base change by a field extension $k\subset K$.

\begin{lem}\label{compatible}
Let $X$ and $Y$ be smooth projective variety over $k$.
The following hold.
\begin{itemize}
\item[(1)] For an object  $E \in \mathrm{Coh}(X)$,  we have $(E^\vee)_K \simeq (E_K)^\vee$.
\item[(2)] For an object $E, F \in \D^b(X)$, we have $(E \otimes F)_K \simeq E_K \otimes F_K$.
\item[(3)] For an object $\E \in \D^b(X \times Y)$, we have $(-)_K \circ \Phi_\E \simeq \Phi_{\E_K} \circ (-)_K$.
\item[(4)] For an object $E \in \D^b(X)$, we have a natural isomorphism 
\[\mathbf{R}\Gamma(X, E) \otimes_k K \iso \mathbf{R}\Gamma(X_K, E_K) \]
\end{itemize}
\end{lem}
\begin{proof}

(1) Take an object $E \in \D^b(X)$. 
Recall the Cartesian diagram.
\[\xymatrix{X_K \ar[d]_{} \ar[r]^{v} & X \ar[d]^{} \\
\mathrm{Spec}(K) \ar[r]^{} & \mathrm{Spec}(k) } \]
We have 
\begin{eqnarray*}
(E^\vee)_K&=& v^*(E^\vee)\\
&\simeq& v^*\mathbf{R}\mathcal{H}om(E, \mathcal{O}_X)     \\
&\simeq& \mathbf{R}\mathcal{H}om(v^*E, \mathcal{O}_{X_K})\\
&=& (E_K)^\vee. 
\end{eqnarray*}

\noindent
(2) This is immediately deduced from the compatibility between the pull-back and the tensor product.

\noindent
(3) Take an object $\E \in \D^b(X \times Y)$.
Consider the Cartesian diagrams.
\[\xymatrix{
&X_K \ar[r]^{v} \ar[d]_{\delta}  & X \ar[d]^{T\varepsilon}&&&X_K \times Y_K \ar[r]^{w} \ar[d]_{} & X \times Y \ar[d]^{}\\
&\mathrm{Spec}(K) \ar[r]^{} & \mathrm{Spec}(k) && &\mathrm{Spec}(K) \ar[r]^{} & \mathrm{Spec}(k)
}\]
Let $p_K: X_K \times Y_K \to X_K$ and $q_K: X_K \times Y_K \to Y_K$ be projections.
For an object $E \in \D^b(X)$, we have
\begin{eqnarray*}
\Phi_\E(E)_K &=& v^*(\mathbf{R}p_*(q^*E \otimes \E))\\
&\simeq& \mathbf{R}p_{K*}w^*(q^*E \otimes \E)     \\
&\simeq&  \mathbf{R}p_{K*}(w^*q^*E \otimes w^*\E)  \\
&\simeq& \mathbf{R}p_{K*}(q_{K*}v^*E \otimes w^*\E) \\
&=& \Phi_{\E_K}(E_K). 
\end{eqnarray*}
Here, we use flat base changes in the second row and the fourth row.

\noindent
(4) It follows from the flat base change theorem.
\end{proof}

The categorical entropy and the Hochschild (co)homological entropy of Fourier--Mukai functors are invariant under base changes. 
\begin{thm}\label{base-change}
For an object $\E \in \D^b(X\times X)$, we have 
\[ h_t(\Phi_\E)=h_t(\Phi_{\E_K})  \]
for any real number $t \in \RR$, and 
\[ h_{HH^*}(\widetilde{\Phi_\E})=h_{HH^*}(\widetilde{\Phi_{\E_K}}),
\hspace{0.2cm}
h_{HH_*}(\widetilde{\Phi_\E})=h_{HH_*}(\widetilde{\Phi_{\E_K}}). 
\]
\end{thm}
\begin{pf}
Let $\mathcal{O}_X(1)$ be a very ample line bundle on $X$. Then $G:=\bigoplus_{i=0}^{\mathrm{dim}X}\mathcal{O}_X(i)$ is a split generator of $\D^b(X)$.
By Remark \ref{split generator}, the object $G_K$ is a split generator of $\D^b(X_K)$.
By Lemma \ref{compatible}, we obtain
\begin{eqnarray*}
\mathbf{R}\mathrm{Hom}(G, \Phi^n_{\E}(G))\otimes_k K &\simeq& \mathbf{R}\Gamma(X,\Phi^n_{\E}(G) \otimes G^\vee)\otimes_k K \\
&\simeq&  \mathbf{R}\Gamma(X_K, (\Phi^n_{\E}(G) \otimes G^\vee)_K)  \\
&\simeq&   \mathbf{R}\Gamma(X_K, \Phi^n_{\E_K}(G_K) \otimes (G_K)^\vee) \\
&\simeq& \mathbf{R}\mathrm{Hom}(G_K, \Phi^n_{\E_K}(G_K)). 
\end{eqnarray*}
Therefore, we have
\[\delta'_t(G, \Phi^n_\E(G))=\delta'_t(G_K, \Phi^n_{\E_K}(G_K)) \]
for any integer $n \in \ZZ_{>0}$.
By Theorem \ref{entropy ext}, we have $h_t(\Phi_\E)=h_t(\Phi_{\E_K})$.

By Lemma \ref{compatible} (4) and Remark \ref{HH via FM}, we obtain
\begin{eqnarray*}
HH^*(\widetilde{\Phi_{\E_K}}) &\simeq& \bigoplus_{i \in \ZZ}\Ext^i(\E_K,\Delta_*\mathcal{O}_{X_K}[\dim X_K]) \\
&\simeq& \bigoplus_{i \in \ZZ}\Ext^i(\E,\Delta_*\mathcal{O}_X[\dim X])\otimes_k K  \\
&\simeq&  HH^*(\widetilde{\Phi_\E})\otimes_k K,
\end{eqnarray*}
which implies $h_{HH^*}(\widetilde{\Phi_{\E}})=h_{HH^*}(\widetilde{\Phi_{\E_K}})$. 
The remaining equality can be proved similarly.
\end{pf}

\section{Algebraic K3 surface}
In this section, we recall some notions for algebraic K3 surfaces. 

\subsection{Autoequivalence groups}

Let $X$ be an algebraic K3 surface over an algebraically closed field $K$ of characteristic zero. 
Fourier--Mukai kernels induce an action on $HH_*(X):=HH_*(\per_{\mathrm{dg}}(X))$ (\cite[Section 4]{CW}), so we obtain a map 
\[
\Aut(\D^b(X))\to \Aut(HH_*(X));~\Phi\mapsto\Phi^{HH_*}.
\]
We define the two subgroups $\Aut^0(\D^b(X)),~\Aut_{CY}(\D^b(X))$ of $\Aut(\D^b(X))$ by
\begin{eqnarray*}
\Aut^0(\D^b(X))&:=&\left\{\Phi\in\Aut(\D^b(X))\mid
\Phi^{HH_*}=\mathrm{id}_{HH_*(X)}  \right\}\\
\Aut_{CY}(\D^b(X))&:=&\left\{\Phi\in\Aut(\D^b(X))\mid \Phi^{HH_*}|_{HH_2(X)}=\mathrm{id}_{HH_2(X)} \right\}. 
\end{eqnarray*}
An element in $\Aut_{CY}(\D^b(X))$ is called a {\it Calabi--Yau autoequivalence}.  
\subsection{Complex model}



Let $K$ be a field of characteristic zero, and $X$ a K3 surface over $K$. 
We can define $X$ using only a finite number of elements of $K$, so there exists a finitely-generated field $\QQ\subset k\subset K$ and a variety $X'$ over $k$ such that $X\simeq X'\times_{k}K$. 
We can embed $k$ in $\CC$, so we obtain a variety $X_\CC \simeq X'\times_{k} \CC$. 
Using the fact that flat base change commutes with 
cohomology, one can show that $X'$, and therefore $X_\CC$, are also K3 surfaces. 
We call $X_\CC$ a {\it complex model} of $X$. 
Applying the pull-back functor $(~)_K$, 
via the extension $k\hookrightarrow K$, to Fourier--Mukai kernels, one has a morphism
\[
(~)_K:~\Aut(\D^b(X'))\to\Aut(\D^b(X));~\Phi_\E\mapsto \Phi_{\E_K}.
\]
\begin{lem}[{\cite[Chapter 16, Section 4.2]{Huy2}, see also \cite[Lemma 6.7]{SS2}}]\label{cpx-model-aut}
Let $X$ be an algebraic K3 surface over an algebraically closed field $K$ of characteristic zero. 
Then base change by the extensions $\CC\hookleftarrow \bar{k}\hookrightarrow K$ yields the isomorphisms
\begin{equation*}
\alpha:~\Aut(\D^b(X_\CC))\xleftarrow[(~)_\CC]{\sim}\Aut(\D^b(X'\times_{k}\bar{k}))\xrightarrow[(~)_K]{\sim}\Aut(\D^b(X));~\Phi\mapsto\alpha(\Phi)
\end{equation*}
Moreover, these isomorphisms induce
\begin{equation*}
\Aut^0(\D^b(X_\CC))\simeq\Aut^0(\D^b(X))\text{ and }
\Aut_{CY}(\D^b(X_\CC))\simeq\Aut_{CY}(\D^b(X)).
\end{equation*}
\end{lem}

\subsection{Complex K3 surface and Mukai lattice}
Let $X$ be a complex algebraic K3 surface. 
The integral cohomology group $H^*(X,\ZZ)$ of $X$ has the lattice structure given by the Mukai pairing 
\begin{equation*}\label{Mukai pairing}
((r_1,c_1,m_1), (r_2,c_2,m_2)):=c_1 \cdot c_2-r_1m_2-r_2m_1
\end{equation*}
for $(r_1, c_1, m_1), (r_2,c_2,m_2) \in H^*(X,\ZZ)$.
The lattice $H^*(X,\ZZ)$ called the {\it Mukai lattice} of $X$ is an even unimodular lattice of signature $(4,20)$.
The Mukai lattice has a weight two Hodge structure $\widetilde{H}(X)$ given by
\begin{equation*}
\widetilde{H}^{2,0}(X):=H^{2,0}(X),~
\widetilde{H}^{1,1}(X):=\bigoplus_{p=0}^{2}H^{p,p}(X),~
\widetilde{H}^{0,2}(X):=H^{0,2}(X).
\end{equation*}
The integral $(1,1)$-part of $\widetilde{H}(X)$ denoted by $\NN(X)$ 
is equal to $H^0(X,\ZZ)\oplus \mathrm{NS}\oplus H^4(X,\ZZ)$. 
The lattice $\NN(X)$ is called the {\it algebraic Mukai lattice}. 

For an object $E \in \D^b(X)$, the {\it Mukai vector} $v(E) \in H^{2*}(X,\QQ)$
of $E$ is given by
\[
v(E)
:=\ch(E)\sqrt{\mathrm{td}_X}
=(\mathrm{rk}(E), c_1(E), \chi(E)-\mathrm{rk}(E)).
\]
By the Riemann--Roch formula, we have the isomorphism $v:K_{\mathrm{num}}(X) \iso \NN(X)$ satisfying
$(v(E),v(F))=-\chi(E,F)$ for any objects $E,F \in \D^b(X)$.

For any autoequivalence $\Phi_\E\in\Aut(\D^b(X))$, 
we define the {\it cohomological Fourier--Mukai transform} $\Phi^H_{\E}:H^{*}(X,\ZZ) \iso H^{*}(X,\ZZ)$ associated to $\Phi_\E$ by
\[
\Phi^H_\E(v):=p_*(q^*(v)\cdot v(\E)), 
\]
which is a Hodge isometry of $\widetilde{H}(X)$. 

The actions of $\Aut(\D^b(X))$ on $HH_*(X)$ and $H^*(X,\CC)$ are compatible with the Hochschild--Kostant--Rosenberg isomorpshim by \cite[Theorem 1.2]{MS} (see also \cite[Appendix A]{BB}). 
we thus have 
\begin{eqnarray*}
\Aut^0(\D^b(X))&=&\{\Phi \in \Aut(\D^b(X)) \mid \Phi^H=\mathrm{id}_{H^*(X,\ZZ)} \}\\
\Aut_{CY}(\D^b(X))&=&\{\Phi \in \Aut(\D^b(X)) \mid \Phi^H|_{\widetilde{H}^{2,0}(X)}=\mathrm{id}_{\widetilde{H}^{2,0}(X)} \}.
\end{eqnarray*}

\section{Symplectic K3 surface}
{\it Symplectic K3 surface} $(\check{X},\omega)$ means a complex K3 surface $\check{X}$ with a K\"ahler form $\omega$ on $\check{X}$. 
We collect several facts on symplectic mapping class groups and Fukaya categories from \cite{SS1,SS2}. 

\subsection{Symplectic mapping class group}
Let $(\check{X},\omega)$ be a symplectic K3 surface. 
We denote ${\rm Symp}(\check{X},\omega)$ the group of symplectomorphisms of $(\check{X},\omega)$ equipped with the $C^\infty$-topology. 
The group $\pi_0{\rm Symp}(\check{X},\omega)$ is called the {\it symplectic mapping class group} of $(\check{X},\omega)$. 
Since $H^1(\check{X},\ZZ)=0$, the connected component containing the identity ${\rm Symp}_0(\check{X},\omega):={\rm Symp}(\check{X},\omega)\cap{\rm Diff}_0(\check{X})$ is exactly the normal subgroup ${\rm Ham}(\check{X},\omega)$ consisting of Hamiltonian diffeomorphisms of $(\check{X},\omega)$. 
We thus have 
\[
\pi_0{\rm Symp}(\check{X},\omega)
\simeq{\rm Symp}(\check{X},\omega)/{\rm Ham}(\check{X},\omega)
\]
The {\it symplectic Torelli subgroup} $I(\check{X},\omega)\subset\pi_0{\rm Symp}(\check{X},\omega)$ is the kernel of the natural homomorphism
\[
\pi_0{\rm Symp}(\check{X},\omega)\to\Aut(H^*(\check{X},\ZZ)).
\]
We call an element in $I(\check{X},\omega)\subset\pi_0{\rm Symp}(\check{X},\omega)$ a {\it symplectic Torelli mapping class}. 

Let $(N,(\ ,\ ))$ be a lattice of signature $(2,t)$. 
The {\it periodic domain} associated with $N$ is 
\begin{equation*}
\Omega(N):=\{\Omega\in\PP(N\otimes_\ZZ\CC)\ |\ (\Omega,\Omega)=0,\ (\Omega,\bar{\Omega})>0\}. 
\end{equation*}
The group of lattice automorphisms $\Aut(N)$ acts on $\Omega(N)$. 
The periodic domain $\Omega(N)$ has two connected components distinguished by the orientation of the positive definite two-plane $(\langle {\rm Re}\Omega,{\rm Im}\Omega\rangle_\RR,(\ ,\ ))$, and interchanged by complex conjugation. 
We denote them by $\Omega^\pm(N)$. 
We denote the set of ($-2$)-classes in $N$ by $\Delta(N):=\{\delta\in N|(\delta,\delta)=-2\}$. 
We define
\begin{equation*}
\Omega^\pm_0(N):=\Omega^\pm(N)\backslash \bigcup_{\delta\in\Delta(N)}
\PP((\delta\otimes_\ZZ\CC)^{\perp}). 
\end{equation*}
Let $\Gamma(N)\subset\Aut(N)$ be the subgroup of lattice automorphisms acting on the discriminant group $N^*/N$. 
The group $\Gamma^+(N)\subset\Gamma(N)$ is defined to be the subgroup of lattice automorphisms preserving $\Omega^+(N)$. 
We define the quotient stack
\begin{equation*}
{\mathcal M}_0(N):=[\Omega^+_0(N)/\Gamma^+(N)]. 
\end{equation*}
Let $\pi_1({\mathcal M}_0(N))$ be the stacky fundamental group of ${\mathcal M}_0(N)$ 
(see \cite[subsection 2.1 and 2.2]{SS2}). 

We define the lattice 
\begin{equation*}
\NN(\check{X},\omega):=[\omega]^{\perp}\cap H^2(\check{X},\ZZ)
\end{equation*}
equipped with the cup product pairing, 
and suppose that $\NN(\check{X},\omega)$ has signature $(2,t)$. 
The {\it complex moduli space} ${\mathcal M}_{\mathrm{cpx}}(\check{X},\omega)$ of $(\check{X},\omega)$ is defined to be 
\begin{equation*}
{\mathcal M}_{\mathrm{cpx}}(\check{X},\omega):=
{\mathcal M}_0(\NN(\check{X},\omega)). 
\end{equation*}
We can construct the symplectic monodromy homomorphism (\cite[Proposition 3.9]{SS2}): 
\begin{equation*}
\pi_1({\mathcal M}_{\mathrm{cpx}}(\check{X},\omega))\to\pi_0{\rm Symp}(\check{X},\omega).
\end{equation*}
\subsection{Fukaya category}
Let $\Lambda$ denote the universal Novikov field over $\CC$: 
\begin{equation*}
\Lambda:=\left\{\sum^{\infty}_{j=0}c_jq^{\lambda_j}~\middle|~c_j\in\CC,\lambda_j\in\RR, \lim_{j\to\infty}\lambda_j=\infty\right\}. 
\end{equation*}
The field $\Lambda$ is an algebraically closed field extension of $\CC$, with a non-Archimedean valuation ${\rm val}:\Lambda\to\RR\cup\{\infty\}$ defined by 
$
{\rm val}\left(\sum^{\infty}_{j=0}c_jq^{\lambda_j}\right)
:=\min_{j}\{\lambda_j~|~c_j\neq0\}.
$
We can define a $\Lambda$-linear $\ZZ$-graded c-unital non-curved $A_{\infty}$ category ${\rm Fuk}(\check{X},\omega)$ called {\it Fukaya category} of a symplectic K3 surface $(\check{X},\omega)$ (see \cite[subsection 2.5]{SS1} and \cite[subsection 5.1]{SS2}).
The split-closed derived category of ${\rm Fuk}(\check{X},\omega)$ is denoted by ${\mathfrak D}{\rm Fuk}(\check{X},\omega)$ (see \cite[Ch.1 (4c)]{Sei}), 
and set $\D{\rm Fuk}(\check{X},\omega):=H^0({\mathfrak D}{\rm Fuk}(\check{X},\omega))$. 
The (graded) symplectic mapping class group of $(\check{X},\omega)$ naturally acts on ${\mathfrak D}{\rm Fuk}(\check{X},\omega)$, so that we have the homomorphism (see \cite[subsection 5.6]{SS2}): 
\begin{equation}\label{sMCG-Fuk}
\pi_0{\rm Symp}(\check{X},\omega)\to\Aut_{CY}({\mathfrak D}{\rm Fuk}(\check{X},\omega))/[2];~~\phi\mapsto\phi_* 
\end{equation}
where $\Aut_{CY}({\mathfrak D}{\rm Fuk}(\check{X},\omega))$ is the subgroup of $\Aut({\mathfrak D}{\rm Fuk}(\check{X},\omega))$ (cf. (\ref{def-aut})) consisting of $F\in \Aut({\mathfrak D}{\rm Fuk}(\check{X},\omega))$ which preserves a 2-Calabi--Yau structure on ${\mathfrak D}{\rm Fuk}(\check{X},\omega)$ (see \cite[subsection 5.3, 5.4 and Lemma 5.11]{SS2}). 

\section{Homological mirror symmetry}
\subsection{Example: the mirror quartic}
Let $\check{X}_0$ be the complex K3 surface called {\it mirror quartic}, defined as the crepant resolution of the quotient of the Fermat quartic hypersurface in $\PP^3_{\CC}$ by $\ZZ/4\times\ZZ/4$. 
We define the set
\[
\Xi_0:=\left\{\kappa=(\kappa_1,\kappa_2,\kappa_3,\kappa_4)\in(\ZZ_{\ge0})^4~\middle|~\sum_i\kappa_i=4\text{ and at least two of the }\kappa_i\text{ are }0\right\}.
\]
Associated to any $\lambda\in(\RR_{>0})^{\Xi_0}$, 
there is a K\"ahler form $\omega_\lambda$ on $\check{X}_0$
(\cite[subsection 1.3]{SS1} for details). 
For any $d\in\Lambda^{\Xi_0}$, 
let $W_d(z_1,z_2,z_3,z_4)\in\Lambda[z_1,z_2,z_3,z_4]$ be the weighted homogeneous polynomial of degree $4$ defined by
\[
W_d(z_1,z_2,z_3,z_4):=-z_1z_2z_3z_4+\sum_{\kappa\in\Xi_0}d_{\kappa}~z_1^{\kappa_1}z_2^{\kappa_2}z_3^{\kappa_3}z_4^{\kappa_4}. 
\]
We define the quartic K3 surface $X_d$ over $\Lambda$ by $X_d:=\{W_d=0\}\subset\PP^3_\Lambda$.  
\begin{thm}[{\cite[Theorem C]{SS1} and \cite[Proposition 6.13 and Example 7.1]{SS2}}]\label{HMS-quartic}
There exists $\lambda\in(\RR_{>0})^{\Xi_0}$ and $d\in\Lambda^{\Xi_0}$ with $\lambda_{\kappa}={\rm val}(d_{\kappa})$ for all $\kappa\in\Xi_0$, satisfying the followings: 
\begin{enumerate}
\item
$X_d$ has Picard rank one. 
\item
We have 
a quasi-equivalence of $\ZZ$-graded $\Lambda$-linear c-unital $A_{\infty}$ categories
\[
{\mathfrak D}{\rm Fuk}(\check{X}_0,\omega_\lambda)\simeq\per_{\mathrm{dg}}(X_d).
\]
\end{enumerate}
\end{thm}
\subsection{Symplectic Torelli mapping class of positive categorical entropy}
We firstly recall an example of an autoequivalence not satisfying the Gromov--Yomdin type equality due to the second author. 

Let $X$ be a complex algebraic K3 surface, and $\O_X(1)$ a very ample line bundle on $X$ such that $c_1(\O_X(1))^2=2d~(d\geq2)$. 
Define
\[
\Phi_0:=(T_{\O_X}\circ(-\otimes\O_X(-1)))^4\in\Aut(\D^b(X)),
\]
where $T_{\O_X}$ is the spherical twist with respect to $\O_X$ .
We then have the following inequality (\cite[Proposition 4.3]{Ouc})
\begin{equation}\label{Ouc-counterexamp}
h_{\mathrm{cat}}(\Phi_0)>\log\rho(K_{\mathrm{num}}(\Phi_0))=\log\rho(\Phi_0^H)=\log\rho((\Phi_0)^{HH_*})(\ge0), 
\end{equation}


For simplicity, we define 
\[
h_{\mathrm{cat}}(\phi):=h_{\mathrm{cat}}(H^0(\phi_*))
\]
for any $\phi\in\pi_0{\rm Symp}(\check{X},\omega)$ via the morphism from (\ref{sMCG-Fuk})
\[
\pi_0{\rm Symp}(\check{X},\omega)
\xrightarrow[(\ref{sMCG-Fuk})]{}\Aut_{CY}({\mathfrak D}{\rm Fuk}(\check{X},\omega))/[2]
\xrightarrow[H^0]{}\Aut(\D{\rm Fuk}(\check{X},\omega))/[2].
\]
Note that $h_{\mathrm{cat}}([\Phi])$ for $[\Phi]\in\Aut(\D{\rm Fuk}(\check{X},\omega))/[2]$ is well-defined by $h_{\mathrm{cat}}(\Phi\circ[2m])=h_{\mathrm{cat}}(\Phi)$ for all $m\in\ZZ$. 

The following is one of the main theorems in this paper. 
\begin{thm}\label{exist-sMCG}
Let $(\check{X_0},\omega_\lambda)$ be the mirror quartic as in Theorem \ref{HMS-quartic}. 
Then there exists a symplectic Torelli mapping class $\phi$ of $(\check{X_0},\omega_\lambda)$ such that 
\begin{equation*}
h_{\mathrm{cat}}(\phi)>\log\rho(H^2(\phi))=0. 
\end{equation*}
\end{thm}
\begin{pf}
Let $X=X_d$ be the quartic K3 surface over $\Lambda$ homologically mirror to $(\check{X},\omega)=(\check{X_0},\omega_\lambda)$ as in Theorem \ref{HMS-quartic}, 
and $X_\CC$ be a complex model of $X$ which is also a quartic over $\CC$. 

By 
(\ref{Ouc-counterexamp}), 
we have $h_{\mathrm{cat}}(\Phi_0)>0$. 
It follows that $\Phi_0$ acts on $H^*(X_\CC,\ZZ)(\simeq HH_*(X_\CC))$ trivially (cf. \cite[Proof in Proposition 4.1]{Fan}) 
from the equivalence to the category of graded matrix factorizations
due to Orlov (\cite[Theorem 3.11]{Orl3}). 
Recall that there are group isomorphisms (cf. Lemma \ref{cpx-model-aut})
\[
\alpha:~\Aut(\D^b(X_\CC))\xleftarrow[(~)_\CC]{\sim}\Aut(\D^b(X'\times_{k}\bar{k}))\xrightarrow[(~)_\Lambda]{\sim}\Aut(\D^b(X));~\Phi\mapsto\alpha(\Phi)
\]
Note that 
the categorical entropy is invariant under base changes for field extensions by Theorem \ref{base-change}, and that the above isomorphism induces $\Aut^0(\D^b(X_\CC))\simeq\Aut^0(\D^b(X))$ (cf. Lemma \ref{cpx-model-aut}). 
We therefore have 
$h_{\mathrm{cat}}(\alpha(\Phi_0))>0$ and $\alpha(\Phi_0)\in\Aut^0(\D^b(X))$. 

Using the Bayer--Bridgeland theorem \cite{BB}
and the symplectic monodromy, 
Sheridan--Smith proved that the morphism obtained by the restriction of (\ref{sMCG-Fuk})
\[
I(\check{X},\omega)
\xrightarrow[(\ref{sMCG-Fuk})]{}\Aut^0({\mathfrak D}{\rm Fuk}(\check{X},\omega))/[2]
\xrightarrow[\text{HMS}]{\sim}\Aut^0(\per_{\mathrm{dg}}(X))/[2]
\xrightarrow[H^0]{\sim}\Aut^0(\D^b(X))/[2]
\]
is surjective \cite[The formula (52) in Proposition 7.8]{SS2}. 
Therefore there exists a symplectic mapping class $\phi$ in the symplectic Torelli subgroup $I(\check{X},\omega)$. 
Clearly, $\phi$ acts on $H^2(\check{X},\ZZ)$ trivially, so that $\log\rho(H^2(\phi))=0$.  
By the definition of 
$h_{\mathrm{cat}}(\phi)$, we have 
$h_{\mathrm{cat}}(\phi)>0$, which completes the proof. 
\qed
\end{pf}
\section{Outlook}


Associated to a symplectic mapping class of a symplectic K3 surface,  
there is the asymptotic invariant called the {\it Floer-theoretic entropy}, which is defined to be the exponential growth rate of the dimensions of its fixed-point Floer cohomology. 
In this section, we introduce the Floer-theoretic entropy and give a question about the relation to the Hochschild entropy and a strategy to prove it. 
\subsection{Floer-theoretic entropy}
For simplicity, we suppose that $(\check{X},\omega)$ is a symplectic K3 surface. 
For a (graded) symplectic mapping class $\phi$, 
we can define a $\ZZ$-graded $\Lambda$-linear space $HF^*(\phi)$ called the {\it fixed-point Floer cohomology} 
of $\phi$ (\cite{DS}, see also \cite[Subsection (3a)]{Sei2}). 
\begin{dfn}[{\cite[Section 6]{Fel2} and \cite[Subsection 1.3, 2.4]{Smi1}}]
For a symplectic mapping class $\phi\in\pi_0{\rm Symp}(\check{X},\omega)$, 
the {\it Floer-theoretic entropy} of $\phi$ is defined by
\[
h_{\mathrm{Floer}}(\phi)
:=\limsup_{n\to\infty}\frac{1}{n}\log\dim_\Lambda HF^*(\phi^n). 
\]
\end{dfn}
There are a few examples whose Floer-theoretic entropy can be computed. 
As an application, Smith used the Floer-theoretic entropy
to show the faithfulness of a representation on symplectic mapping class groups of representation varieties (\cite[Subsection 1.3, 2.4]{Smi1}). 
The Floer-theoretic entropy is also closely related to the radius of convergence of the symplectic zeta function (\cite{Fel1,Fel2,Fel3,Fel4}). 
\subsection{Relation to the Hochschild entropy}
Let $(\check{X},\omega)$ be a symplectic K3 surface. 
For a symplectic mapping class $\phi\in\pi_0{\rm Symp}(\check{X},\omega)$, there are two notions of entropy i.e. the Floer-theoretic entropy and the categorical entropy, which are expected to be related via the homological mirror equivalence. 
As in Theorem \ref{HH-cat}, 
the categorical entropy is equal to or greater than the Hochschild entropy, so we then have to compare the Floer-theoretic entropy and the Hochschild entropy. 
\begin{question}\label{Floer-HH}
Let $(\check{X},\omega)$ be a symplectic K3 surface. 
For a symplectic mapping class $\phi\in\pi_0{\rm Symp}(\check{X},\omega)$, 
does the equation
\[
h_{\mathrm{Floer}}(\phi)=h_{HH^*}(\phi_*)
\]
hold?
\end{question}
All elements of finite order clearly satisfy the above equation. 
We hope that Question \ref{Floer-HH} is affirmative in general, and here propose a strategy to prove it. 
Firstly, two entropy is defined as the asymptotics of dimensions of some $\Lambda$-linear spaces:
\[
HF^*(\phi)~and~
HH^*(\phi_*)
\]
so it suffices to show linear isomorphisms between them as follows. 

\vspace{0.3cm}

\begin{enumerate}
\item[{\it Step~1}.]
For a symplectomorphism $\psi\in{\rm Symp}(\check{X},\omega)$, 
its graph $\Gamma_\psi:=\{(x,\psi(x))\}$
is a Lagrangian submanifold of $(\check{X}\times \check{X}, \mathrm{pr}_1^*\omega-\mathrm{pr}_2^*\omega)$, where $\mathrm{pr}_i: \check{X}\times \check{X}\to \check{X}$ is the projection onto the $i$-th component. 
Let $\Delta:=\Gamma_{\mathrm{id}_{\check{X}}}$ be the diagonal. 
Then for a symplectic mapping class $\phi:=[\psi]\in\pi_0{\rm Symp}(\check{X},\omega)$, show
\[
HF^*(\phi)
\simeq HF^*(\Gamma_\psi, \Delta)
:=H^*(CF(\Gamma_\psi,\Delta)).
\]
This is true for the case of non-degenerate fixed points (\cite[Example 3.8]{Sei2}).


\vspace{0.2cm}

\item[{\it Step~2}.]
Let $\mathfrak{F}:={\rm Fuk}(\check{X},\omega)$. 
Show
\[
H^*(CF(\Gamma_\psi,\Delta))
\simeq H^*(hom_{fun(\mathfrak{F}, \mathfrak{F})}(\phi_*, {\rm id}_\mathfrak{F})), 
\]
where we recall that $fun(\mathfrak{F}, \mathfrak{F})$ is the $A_\infty$ category of c-unital $A_\infty$ endofunctor of $\mathfrak{F}$ (see \ref{A_infty-ver}). 
To show, we need to consider the fully-faithfulness of the Ma'u--Wehrheim--Woodward-type functor (\cite[Theorem 1.1]{MWW}), which is related to whether the diagonal 
$\Delta$ is split-generated by products of Lagrangians in $\check{X   }$. 
Fukaya proved the existence of some variant of this functor in a more general setting (\cite[Corollary 7.4]{Fuk}). 

For monotone symplectic manifolds, 
Smith showed this isomorphism for a symplectomorphism represented by compositions of Dehn twists in some vanishing cycles (\cite[Corollary 3.12]{Smi1}). 


\vspace{0.2cm}

\item[{\it Step~3}.]
Let
$\mathfrak{DF}:={\mathfrak D}{\rm Fuk}(\check{X},\omega)$. Show
\[
H^*(hom_{fun(\mathfrak{F},\mathfrak{F})}(\phi_*, {\rm id}_\mathfrak{F}))
\simeq H^*(hom_{fun(\mathfrak{DF}, \mathfrak{DF})}(\phi_*, {\rm id}_\mathfrak{DF}))
(\simeq HH^*(\phi_*)). 
\]
When $\phi=[{\rm id}_{\check{X}}]$, this isomorphism is a well-known fact that the Hochschild cohomology of an $A_\infty$ category is invariant under passing to a split-closed triangulated envelope. 




\end{enumerate}

\vspace{0.5cm}

\noindent
If $h_{HH^*}(\widetilde{\Phi_0})=h_{\mathrm{cat}}(\Phi_0)$ holds and Question \ref{Floer-HH} is affirmative for the symplectic mapping class $\phi$ appeared in Theorem \ref{exist-sMCG}, 
then $\phi$ gives a Floer-theoretic and higher-dimensional analogue of (real) surface's Torelli mapping classes of positive topological entropy 
(\cite[Corollary in Section 6]{Thu}). 


\end{document}